\documentclass[a4paper,11pt]{article}
\usepackage[a4paper, total={6.5in, 10in}]{geometry}
\usepackage{amsfonts,amsmath,amssymb,amsthm,enumerate,bm,}

\allowdisplaybreaks[4] 
\theoremstyle{plain}
\newtheorem{theorem}{\noindent\bf Theorem}[section]

\newtheorem{lemma}[theorem]{\noindent\bf Lemma}

\theoremstyle{remark}
\newtheorem{remark}[theorem]{\noindent\bf Remark}
\theoremstyle{definition}
\newtheorem{definition}[theorem]{\noindent\bf Definition}
\numberwithin{equation}{section}

\def\be{\begin{eqnarray}}%%
\def\ee{\end{eqnarray}}%%
\def\ben{\begin{eqnarray*}}%%
\def\een{\end{eqnarray*}}%%
\def\benum{\begin{enumerate}}%%
\def\eenum{\end{enumerate}}%%
\newcommand{\lr}{\left(}
\newcommand{\rr}{\right)}

\title{\bf \Large Rubio de Francia Extrapolation Theorems \\for Quasi-Monotone Functions}
\author{Arun Pal Singh$^1$, Rahul Panchal$^2$, Pankaj Jain$^3$ and Monika Singh$^4$} 

\date{}

\begin{document}
\maketitle

\indent  \emph{$^1$Department of Mathematics,
Dyal Singh College (University of Delhi),
Lodhi Road, \\  
\indent ~~Delhi - 110003, INDIA}~~
Email : arunpalsingh@dsc.du.ac.in \\ 

\indent  \emph {$^2$Department of Mathematics,
University of Delhi, Delhi - 110007, INDIA} \\
\indent ~~Email: drrpanchal0@gmail.com \\

\indent  \emph {$^3$Department of Mathematics, 
South Asian University,
Akbar Bhawan,  Chanakya Puri, \\
\indent ~~New Delhi - 110021, INDIA}~~
Email :  pankaj.jain@sau.ac.in; pankajkrjain@hotmail.com \\

\indent  \emph {$^4$Department of Mathematics
Lady Shri Ram College for Women (University of Delhi), \\
\indent ~~Lajpat Nagar, Delhi - 110024, INDIA}~~
Email : monikasingh@lsr.du.ac.in

\bigskip

\begin{abstract}
\noindent We prove Rubio de Francia extrapolation results in Lebesgue and grand Lebesgue spaces for quasi monotone functions with $QB_{\beta,p}$ weights. The extrapolation in Lebesgue spaces with the weight class $QB_{\beta,\infty}$  has also been investigated. As an application, we characterize the boundedness of the Hardy averaging operator for quasi monotone functions in the grand Lebesgue spaces.

\end{abstract}

\thispagestyle{empty}
\section{Introduction}
We shall denote by $\mathcal{M},$ the set of all measurable functions defined and finite almost everywhere (a.e.) on $\mathbb{R}^+.$  Also, $\mathcal{M}^{+}\subset \mathcal{M}$ and $\mathcal{M}^+_{\downarrow}\subset \mathcal{M}^{+}$ will denote, respectively, the cones of non-negative and non-negative non-increasing ($\downarrow$) functions in $\mathcal{M}.$ By a weight $w,$ we mean a function in $\mathcal{M}^{+}$ which is locally integrable as well. For a weight $w$ and $1 \le p<\infty,$ denote by $L_w^p,$ the weighted Lebesgue space consisting of all $f\in \mathcal{M}$ such that
\[
\|f\|_{L_w^p}:= \lr\int_0^\infty |f|^p w\rr^{1/p}<\infty.
\]
A weight $w$ is said to be in the Muckenhoupt class $A_p,~1<p<\infty,$ if
\[ 
[w]_{A_p}:= \sup_J \lr\frac{1}{|J|}\int_J w \rr \lr \frac{1}{|J|} \int_J w^{-p'/p}\rr^{p-1} <\infty, 
\]
and in class $A_1,$ if 
\[ [w]_{A_1}:= \sup_J  \text{ess}\sup_{x\in J} \frac{W(J)}{w(x)|J|}< \infty,\]
where supremum is taken over all non-degenerate intervals $J\subset \mathbb{R}^+,~ \frac{1}{p}+\frac{1}{p'}=1$ and $W(J):=\int_J w(x)dx.$ 
The weight class $A_p$ is found to be useful in so many ways. It characterizes the boundedness of the maximal operator \cite{mu2} and Riesz potential \cite{cg} in Lebesgue spaces. Moreover, this class also characterizes the boundedness of these operators in grand Lebesgue spaces \cite{fgj}, \cite{m}. Another beauty of the $A_p$-class of weights can be realized via the celebrity extrapolation result of J.L. Rubio de Francia \cite{rdf}, which asserts that if a sublinear operator $T$ is bounded on $L_w^{p_0}$ for every $w\in A_{p_0} ~(p_0\ge 1,$ fixed) with the constant of inequality depending only on $[w]_{A_{p_0}},$ then for every $1<p<\infty,~T$ is bounded on $L_w^p$ for every $w\in A_p.$ This extrapolation result was further re-investigated and explored by many people (see \cite{dcump2} and the references therein) and now it is known that the operator $T$ has no role to play. In fact, it is known that  if $(f,g)$ is a pair of non-negative measurable functions such that for some $1\leq p_0<\infty,$ the inequality
\[ \int_0^\infty f^{p_0}(x)w(x)dx\leq C \int_0^\infty g^{p_0}(x)w(x)dx 
\]
holds for every $w\in A_{p_0}$ with constant $C$ depending on $[w]_{A_{p_0}},$ then for every $1<p<\infty,$ the inequality
\[
\int_0^\infty f^p(x)w(x)dx\leq C \int_0^\infty g^p(x)w(x)dx
\]
holds for every $w\in A_p$ with constant $C$ depending on $[w]_{A_p}.$

This theory has been generalized to $A_\infty$-weights also, see \cite{dcump1}. In \cite{cl}, Carro and Lorente established a parallel extrapolation theory for a pair of functions from $\mathcal{M}^+_{\downarrow}$ in the framework of $B_p$-class of weights: A weight $w$ is said to belong to the class $B_p ~(p>0)$  if there exists a constant $C>0$ such that the inequality 
\[
\int_r^\infty \lr\frac{r}{x}\rr^p w(x)dx \leq C\int_0^r w(x)dx
\]
holds for every $r>0$. Like the $A_p$-class of weights, the weight class $B_p$ is also an important class of weights. It characterizes the boundedness of the Hardy averaging operator
\[
Hf(x) := \frac{1}{x}\int_0^x f(t)dt
\] 
in $L_w^p$ spaces for $f\in \mathcal{M}^+_{\downarrow}$ (see \cite{and, s}) and also in grand Lebesgue spaces (defined in Section 3) for $f\in \mathcal{M}^+_{\downarrow}$ \cite{jk, m}. These characterizations, in fact, are equivalent to the boundedness of the maximal operator, respectively in, Lorentz space $\Lambda^p(w)$ \cite{am} and grand Lorentz space $\Lambda^{p)} (w)$ \cite{jk}.

In this paper, we consider quasi non-increasing functions, the class of such functions being denoted by $Q_\beta :$ A function $f\in \mathcal{M^+}$ is said to belong to $Q_\beta, ~\beta\in \mathbb{R},$ if $x^{-\beta}f(x)$ is non-increasing. Clearly $\mathcal{M}^+_{\downarrow} = Q_0.$
For the functions $f \in Q_\beta,$ Bergh, Burenkov and Persson \cite{bbl} investigated Hardy's inequality with power type weights, while for general weights it has been proved in \cite{pma} that the inequality 
\[
\lr \int_0^\infty {\lr \frac{1}{x}\int_0^x f(t)dt \rr}^p w(x)dx\rr \leq C\int_0^\infty f^p(x) w(x)dx, ~1\leq p<\infty 
\]
holds for all $f\in Q_{\beta}$ if and only if $w\in QB_{\beta,p},~\beta >-1,$ i.e.,
\be \label{eqn aa}
\int_r^\infty \lr\frac{r}{x}\rr^p w(x)dx \leq C\int_0^r {\lr \frac{x}{r} \rr}^{\beta p} w(x)dx,~ r>0.
\ee
Note that for $\beta = 0,$ the weight class $QB_{\beta,p}$ reduces to the class $B_p.$ In the present paper, we define a variant of the class $QB_{\beta,p},$ to be denoted by $\widehat{Q}B_{\beta,p},$ and prove the extrapolation results for this class of weights, as well as for the weight class
\[QB_{\beta,\infty}:= \bigcup_{p>0} QB_{\beta,p}~.\]
Further, we prove the extrapolation result for quasi-monotone functions in the frame of grand Lebesgue spaces. As an application, we prove the boundedness of the Hardy averaging operator for quasi-monotone functions in the grand Lebesgue spaces. Our results generalize the extrapolation results of Carro and Lorente \cite{cl} and Meskhi \cite{amksh}. Throughout, all the functions used in this paper are assumed to be non-negative and measurable.

\section{Extrapolation results in Lebesgue spaces}
\noindent
For $p>0,$ we say that a weight $w\in QB_{\beta,\psi,p}$ if 
\be \label{eqn *}
\int_r^\infty \lr \frac{\Psi(r)}{\Psi(x)}\rr^p w(x)\,dx \le C \int_0^r \lr \frac{\Psi(x)}{\Psi(r)}\rr^{\beta p} w(x)~dx,r>0
\ee
for some constant $C>0$ and $\Psi (x):=\int_0^x \psi (t)dt,$ where $\psi$ is a non-negative, non-increasing locally integrable function, i.e., $\psi\in L_{loc}^1.$  For $\psi \equiv 1,$ the weight class $QB_{\beta,\psi,p}$ reduces to the class $QB_{\beta,p}.$\\

In \cite{pma}, the class $QB_{\beta,\psi,p}$ was used to characterize the boundedness of the operator 
\[
S_{\psi}f(x):=\frac{1}{\Psi(x)}\int_0^{x} f(t)\psi (t)dt
\]
on the cone of functions $f\in Q_\beta.$ Precisely, the following was proved :\\

\noindent {\bf Theorem A} \cite{pma}.
\emph {Let $p \geq 1$ and $-1<\beta\leq 0$. Then the inequality
\[
\int_0^\infty {\lr S_\psi f\rr}^p(x)w(x)dx \leq C'\int_0^\infty f^p(x)w(x)dx
\]
holds for all $f\in Q_\beta$ if and only if $w\in QB_{\beta,\psi,p},$ where $C'= \frac{C+1}{{(\beta+1})^p}$ and $C$ is as in \emph {(\ref{eqn *})}.}\\

We define $QB_{\beta,\psi,p}$-constant for a weight  $w \in QB_{\beta,\psi,p}$ as follows
\begin{align} \label{eq23}
[w]_{QB_{\beta,\psi,p}} := & \inf  \left\{ D: \int_r^\infty \lr \frac{\Psi(r)}{\Psi(x)}\rr^p w(x)\,dx \le (D-1) \int_0^r \lr \frac{\Psi(x)}{\Psi(r)}\rr^{\beta p}w(x)dx,r>0 \right\}.
\end{align} 
\begin{remark} \label{rmk21}
Note that
\begin{enumerate}
\item  $[w]_{QB_{\beta,\psi,p}} >1.$ 
\item For $-1<\beta \le 0$ and $p \le q,$ we have
$QB_{\beta,\psi,p} \subset QB_{\beta,\psi,q}.$
\end{enumerate}
\end{remark}
We begin with the following :
\begin{lemma} \label{lemma 0}
Let the function $\varphi$ be non-decreasing $(\uparrow)$ defined on $(0,\infty), ~f,g\in Q_\beta ~(\beta>-1), 0<p_0<\infty,~\psi\in L_{loc}^1$ be $\downarrow$ and $ \displaystyle \lim_{x\rightarrow\infty}\Psi(x)=\infty.$ Suppose that for each $w\in QB_{\beta,\psi,p_0},$ the inequality
\[
\int_0^\infty f(x)w(x)dx \leq \varphi ([w]_{QB_{\beta,\psi,p_0}})\int_0^\infty g(x)w(x)dx
\]
holds. Then for every $0<\varepsilon<p_0({\beta+1})$ and $t>0,$ the following inequality holds:
\[
\int_0^t f(s)(\Psi(s))^{p_0-1-\varepsilon}\psi(s)ds \leq \varphi  \lr\frac{p_0(\beta+1)}{\varepsilon}\rr \int_0^t g(s)(\Psi(s))^{p_0-1-\varepsilon}\psi(s)ds.
\]
\end{lemma}
\begin{proof}

Let $v\in \mathcal{M^+_\downarrow}.$ Set $w(x)= v(x)(\Psi(x))^{p_0-1-\varepsilon}\psi(x)$  so that $w\in L_{loc}^1$.
We claim that $w\in QB_{\beta,\psi,p_0}.$ Indeed, we have

\begin{align*}
\lr \Psi(r)^{\beta+1}\rr^{p_0}\int_r^\infty \frac{w(x)}{\Psi(x)^{p_0}}dx
&= \lr\Psi(r)^{\beta+1} \rr^{p_0}\int_r^\infty v(x)(\Psi(x))^{-1-\varepsilon}\psi(x)dx  \\
&\le \frac{v(r)}{\varepsilon}(\Psi(r))^{p_0(\beta+1) - \varepsilon}  \\
&= \frac{(p_0(\beta+1)-\varepsilon)}{\varepsilon}v(r)\int_0^r (\Psi(x))^{p_0(\beta+1)-1-\varepsilon}\psi(x)dx \nonumber \\
&\le \lr \frac{p_0(\beta+1)}{\varepsilon}-1\rr \int_0^r v(x)(\Psi(x))^{p_0(\beta+1)-1-\varepsilon}\psi(x)dx \label{ineq 1}   \\ 
&\le  \frac{p_0(\beta+1)}{\varepsilon} \int_0^r (\Psi(x))^{\beta p_0}v(x)(\Psi(x))^{p_0-1-\varepsilon}\psi(x)dx \\
&= \frac{p_0(\beta+1)}{\varepsilon}\int_0^r (\Psi(x))^{\beta p_0}w(x)dx. 
\end{align*}
The assertion now follows on taking $v(x)=\chi_{(0,s]}(x)$ and using the fact that $[w]_{QB_{\beta,\psi,p_0}} \leq \frac{p_0(\beta+1)}{\varepsilon}.$
\end{proof}

\begin{definition} \label{definition 1}
For a given $\beta>-1,$ a weight function $w$ is said to be in the class $\widehat{Q}B_{\beta,p}$ if\\
(i) $w\in QB_{\beta,p}$; and \\
(ii) there exists $0 < \varepsilon <p(\beta +1)$ such that $w\in QB_{\beta,p-\varepsilon}.$
\end{definition}

\begin{remark}
The class $\widehat{Q}B_{\beta,p}$ in Definition$\ref{definition 1}$ is reasonably defined. In view of Lemma 2.3 \cite{pma}, it is clear that for $\beta \geq 0,$ $\widehat{Q}B_{\beta,p}=QB_{\beta,p}$. We prove below that for $-1<\beta<0,$ the power weights belong to the class $\widehat{Q}B_{\beta,p}$. It is of interest if the same can be proved for general weights as well.
\end{remark}

\begin{lemma} \label{lemma b}
Let $1\leq p<\infty,~-1<\beta<0$ and $\alpha \in \mathbb{R}$. If $x^\alpha \in QB_{\beta,p},$ then there exists $0<\varepsilon < p(\beta+1)$ such that $x^\alpha \in QB_{\beta,p-\varepsilon}.$
\end{lemma}
\begin{proof}
Since $x^\alpha \in QB_{\beta,p},$ we have that 
\be \label{eqn 1}
\int_r^\infty \lr\frac{r}{x}\rr^p x^\alpha dx\leq C\int_0^r \lr\frac{x}{r}\rr^{\beta p}x^\alpha dx,~ r>0
\ee
which holds if and only if 
\be \label{eqn 2}
-\beta p-1<\alpha<p-1.
\ee
Choose $\varepsilon>0$ such that $0<\varepsilon<p-\alpha-1$. Clearly, $0 < \varepsilon < p(\beta +1).$ Now, using the estimates  $(\ref{eqn 1})$ and $(\ref{eqn 2})$ at appropriate places, we obtain
\begin{align*}
\int_r^\infty \lr \frac{r}{x}\rr ^{p-\varepsilon} x^\alpha dx
&= \frac{r^{\alpha +1}}{p-\varepsilon-\alpha-1}\\
&= \frac{p-\alpha-1}{p-\varepsilon-\alpha-1}\int_r^\infty \lr\frac{r}{x}\rr^p x^\alpha dx\\
& \le C\lr\frac{p-\alpha-1}{p-\varepsilon-\alpha-1}\rr \int_0^r \lr\frac{x}{r}\rr^{\beta p} x^\alpha dx\\
&= \frac{K}{(\alpha+\beta p+1)r^{\beta p}}r^{\beta p+\alpha+1}\\
&= K\lr \frac{\beta(p-\varepsilon)+\alpha+1}{\alpha+\beta p+1}\rr\int_0^r \lr\frac{x}{r}\rr^{\beta(p-\varepsilon)}x^\alpha dx
\end{align*}
i.e., $x^\alpha \in QB_{\beta,p-\varepsilon}$ with the constant
\[
C^* := K\lr\frac{\beta(p-\varepsilon)+\alpha+1}{\alpha+\beta p+1}\rr,
\]
where $K = C\lr \frac{p-\alpha-1}{p-\varepsilon-\alpha-1}\rr$ and $C$ is as in (\ref{eqn 1}).
\end{proof}
\begin{remark}
For $-\beta p-1<\alpha<p-1,$ from Lemma $\ref{lemma b}$ and (\ref{eq23}) it follows that
\[
[x^\alpha]_{QB_{\beta,p-\varepsilon}}\le C^*+1= C\lr \frac{p-\alpha-1}{p-\varepsilon-\alpha-1}\rr \lr\frac{\beta(p-\varepsilon)+\alpha+1}{\alpha+\beta p+1}\rr+1.
\]
\end{remark}
We now prove the first main extrapolation theorem:
\begin{theorem} \label{thm 1}
Let $\varphi \uparrow$ be defined on $(0,\infty),~ (f,g)$ be a pair of functions such that $f,g\in Q_\beta, -1<\beta\leq 0$ and $1\leq p_0<\infty.$ Suppose that for every $w\in {Q}B_{\beta,p_0},$ the inequality
\[
\int_0^\infty f^{p_0}(x)w(x)dx \leq \varphi([w]_{QB_{\beta,p_0}})\int_0^\infty g^{p_0}(x)w(x)dx
\]
holds. Then for all $p_0 \leq p<\infty$ and all $w \in \widehat{Q}B_{\beta,p},$ the following holds:
\[
\int_0^\infty f^p(x)w(x)dx \leq C\int_0^\infty g^p(x)w(x)dx,
\]
where 
\[
C= \displaystyle\inf_{0<\varepsilon<p_0(\beta+1)}[w]_{QB_{\beta,(p_0-\varepsilon)\frac{p}{p_0}}}\left[\frac{1}{\beta+1}\lr \frac{p_0(\beta+1)-\varepsilon}{p_0-\varepsilon}\rr \varphi \lr\frac{p_0(\beta+1)}{\varepsilon}\rr\right]^{p/p_0}.
\]
\end{theorem}
\begin{proof}
The case $\beta=0$ is just Theorem 2.1 of \cite{cl}. So, we assume that $-1<\beta<0.$

Let $p_0\leq p<\infty$,$w \in \widehat{Q}B_{\beta,p}$ and $0<\varepsilon <p_0(\beta+1)$. Clearly the function $h(x):= x^{-\beta}f(x)$ is $\downarrow.$ Note that 
\be \label{eqn 4}
\int_0^\infty f^p(x)w(x)dx = \int_0^\infty h^p(x)w(x)x^{\beta p}dx.
\ee
Since $h$ is decreasing, we have
\[
h^{p_0}(x)\leq \frac{p_0(\beta+1)-\varepsilon}{x^{p_0(\beta+1)-\varepsilon}}\int_0^x h^{p_0}(s)s^{p_0(\beta+1)-\varepsilon-1}ds
\]
which together with $(\ref{eqn 4})$ and Lemma $\ref{lemma 0}$ (for $\psi \equiv 1$) gives
\begin{align}
&\int_0^\infty  f^p(x)w(x)dx \nonumber \\
&\leq \lr\frac{p_0(\beta+1)-\varepsilon}{p_0-\varepsilon}\rr^{p/p_0}  \int_0^\infty \lr\frac{p_0-\varepsilon}{x^{p_0-\varepsilon}} \int_0^x f^{p_0}(s)s^{p_0-1-\varepsilon}ds\rr^{p/p_0}w(x)dx \nonumber \\
&\leq \lr\frac{p_0(\beta+1)-\varepsilon}{p_0-\varepsilon}\rr^{p/p_0}\varphi\lr\frac{p_0(\beta+1)}{\varepsilon}\rr^{p/p_0} \int_0^\infty\lr\frac{p_0-\varepsilon}{x^{p_0-\varepsilon}}\int_0^x g^{p_0}(s)s^{p_0-1-\varepsilon}ds\rr^{p/p_0}w(x)dx \nonumber \\
&=\gamma \int_0^\infty\lr\frac{p_0-\varepsilon}{x^{p_0-\varepsilon}}\int_0^x g^{p_0}(s)s^{p_0-1-\varepsilon} ds\rr^{p/p_0}w(x)dx  \nonumber \\
&=\gamma\int_0^\infty\lr S_\psi g^{p_0}(x)\rr^{p/p_0}w(x)dx, \label{ineqn z}
\end{align}
where $\psi(s)=s^{p_0-1-\varepsilon}$ and 
\[\gamma = \lr\frac{p_0(\beta+1)-\varepsilon}{p_0-\varepsilon}\rr^{p/p_0}\varphi\lr\frac{p_0(\beta+1)}{\varepsilon}\rr^{p/p_0}.
\]
Now, since $w\in \widehat{Q}B_{\beta,p},$ by definition, there exists $\widetilde{\varepsilon}>0$ such that $w\in QB_{\beta,p-\widetilde{\varepsilon}}.$ It is sufficient to take $\varepsilon$ so that $p-\widetilde{\varepsilon}=(p_0-\varepsilon)\frac{p}{p_0}$ or $\varepsilon=\frac{p_0}{p}\widetilde{\varepsilon}.$ Then $w\in QB_{\beta,(p_0-\varepsilon)\frac{p}{p_0}}$ which gives that for all $r>0$ the following inequality holds:
\[
\int_r^\infty \lr\frac{r}{x}\rr^{(p_0-\varepsilon)\frac{p}{p_0}}w(x)dx \leq (A-1)\int_0^r \lr\frac{x}{r}\rr^{\beta(p_0-\varepsilon)\frac{p}{p_0}}w(x)dx,
\]
or
\[
\int_r^\infty \lr\frac{\Psi(r)}{\Psi(x)}\rr^{p/p_0}w(x)dx \leq (A-1)\int_0^r \lr\frac{\Psi(x)}{\Psi(r)}\rr^{\beta(p/p_0)}w(x)dx
\]
with $\psi(s)=s^{p_0-1-\varepsilon},$ which by Theorem A holds if and only if
\[
\int_0^\infty\lr S_\psi g^{p_0}(x)\rr^{p/p_0}w(x)dx\leq \frac{A}{(\beta+1)^{p/p_0}}\int_0^\infty g^p(x)w(x)dx,
\]
where $A = [w]_{QB_{\beta,(p_0-\varepsilon)\frac{p}{p_0}}}=[w]_{QB_{\beta,p-\widetilde{\varepsilon}}}$.\\

\noindent Consequently, $(\ref{ineqn z})$ gives
\begin{align*}
\int_0^\infty f^p(x)w(x)dx 
&\leq  \frac{\gamma A}{(\beta +1)^{p/p_0}}\int_0^\infty g^p(x)w(x)dx\\
&= K\int_0^\infty g^p(x)w(x)dx,
\end{align*}
where
\[
K = [w]_{QB_{\beta,(p_0-\varepsilon)\frac{p}{p_0}}}\left[ \lr\frac{p_0(\beta+1)-\varepsilon}{(\beta+1)(p_0-\varepsilon)}\rr \varphi\lr\frac{p_0(\beta+1)}{\varepsilon}\rr\right]^{p/p_0}.
\]
Since $\varepsilon \in (0,p_0(\beta+1))$ is arbitrary, taking infimum over all such $\varepsilon$, the assertion follows.
 \end{proof}

In view of the Remark \ref{rmk21} (for $\psi \equiv 1$), following the definition of the class $B_\infty$  \cite{cl}, we define the class $QB_{\beta,\infty}$ as
\[
QB_{\beta,\infty}:= \bigcup_{p>0} QB_{\beta,p}
\]
and we also define 
\[
[w]_{QB_{\beta,\infty}}:= \inf \left \{ [w]_{QB_{\beta,p}}: w\in QB_{\beta,p},~p>0 \right \}.
\]
Similarly, we define 
\[
QB_{\beta,\psi,\infty}:= \bigcup_{p>0}QB_{\beta,\psi,p}
\]
and
\[
1 \le [w]_{QB_{\beta,\psi,\infty}}:= \inf \left \{ [w]_{QB_{\beta,\psi,p}}: w\in QB_{\beta,\psi,p},~p>0 \right \}.
\]

We prove the following:
\begin{lemma} \label{lemma c}
Let $\psi:\mathbb{R}^+\to \mathbb{R}^+$ be $\uparrow,~ v \in L^1_{loc}$ be $\downarrow, ~-1<\beta\leq 0$ and $\alpha>-1.$ Then the function $w$ defined by 
\[
w(x)= \Psi^\alpha (x)\psi (x)v(x)
\]
belongs to the class $QB_{\beta,\psi,\infty}.$
\end{lemma}
\begin{proof}

Let $0<r<\infty$ be arbitrary and choose $p_0$ such that $\alpha +1<p_0<-\frac{1}{\beta}(\alpha+1)$. Then we have
\begin{align*}
\int_r^\infty \lr \frac{\Psi^{\beta+1}(r)}{\Psi(x)}\rr^{p_0}w(x)dx 
& = (\Psi(r))^{(\beta+1)p_0}\int_r^\infty (\Psi(x))^{\alpha-p_0}\psi(x)v(x)dx\\
& \leq \frac{1}{(p_0-\alpha-1)}(\Psi(r))^{\beta p_0+\alpha+1}v(r)\\
& \leq \lr \frac{\beta p_0+\alpha+1}{p_0-\alpha-1}\rr \int_0^r (\Psi(x))^{\beta p_0+\alpha} \psi(x)v(x)dx\\
& = \lr \frac{\beta p_0+\alpha+1}{p_0-\alpha-1}\rr \int_0^r (\Psi(x))^{\beta p_0} w(x)dx
\end{align*}
and the assertion follows. Moreover, $[w]_{QB_{\beta,\psi,\infty}} \le \frac{\beta p_0+\alpha+1}{p_0-\alpha-1}+1.$
\end{proof}
Below, we prove an extrapolation result for $QB_{\beta,\infty}$- class of weights.
\begin{theorem}
Let $\varphi$ be $\uparrow$ defined on $(0,\infty), -1<\beta \leq 0,(f,g)$ be a pair of functions such that $f,g \in Q_\beta$ and $0<p_0<\infty$. Suppose that for every weight $w \in QB_{\beta,\infty},$ the inequality
\be \label{eqn 5}
\int_0^\infty f^{p_0}(t)w(t)dt \leq \varphi([w]_{QB_{\beta,\infty}})\int_0^\infty g^{p_0}(t)w(t)dt
\ee
holds. Then for every $p_0 \le p<\infty$ and $w\in QB_{\beta,\infty}$ the following holds
\[
\int_0^\infty f^p(t)w(t)dt \leq K\int_0^\infty g^p(t)w(t)dt ,
\]
with 
\[
K=\inf_{\alpha>-1} [w]_{QB_{\beta,\frac{(\alpha+1)p}{p_0}}} \lr \frac{\varphi(1)}{\beta+1}\rr^{p/p_0}. 
\]

\end{theorem}

\begin{proof}
For $s>0$ and $\alpha>-1,$ consider the following
\[
\widetilde{w}(t)= \chi_{(0,s)}(t)t^\alpha.
\]
Clearly by Lemma $\ref{lemma c},~ \widetilde{w} \in QB_{\beta,\infty}.$ Then, in view of Remark \ref{rmk21} and Lemma \ref{lemma c}, we have
\[1 \le [\widetilde{w}]_{QB_{\beta, \psi, \infty}} \le \displaystyle \lim_{p_0\rightarrow\infty}\frac{\beta p_0+\alpha+1}{p_0-\alpha-1} +1 = \beta +1 \le 1\]
and consequently, in view of $(\ref{eqn 5})$ the following holds
\be \label{eqn 6}
\int_0^s f^{p_0}(t)t^\alpha dt \leq \varphi(1)\int_0^s g^{p_0}(t)t^\alpha dt.
\ee
Since $t^{-\beta}f(t)$ is $\downarrow$, we find that
\begin{align*}
f^{p_0}(t)
&=\frac{\alpha+1}{t^{\alpha+1}}\int_0^t f^{p_0}(t)s^\alpha ds\\
&= \frac{\alpha+1}{t^{\alpha+1}}\int_0^t (t^{-\beta}f(t))^{p_0} t^{\beta p_0}s^\alpha ds\\
& \leq \frac{\alpha+1}{t^{\alpha+1}}\int_0^t (s^{-\beta}f(s))^{p_0} t^{\beta p_0}s^\alpha ds\\
&= \frac{\alpha+1}{t^{\alpha+1}}\int_0^t f^{p_0}(s)\lr\frac{t}{s}\rr^{\beta p_0} s^\alpha ds\\
& \leq \frac{\alpha+1}{t^{\alpha+1}}\int_0^t f^{p_0}(s)s^\alpha ds
\end{align*}
which in view of $(\ref{eqn 6})$ gives
\begin{align}
\int_0^\infty f^p(t)w(t)dt
& \leq \int_0^\infty\lr\frac{\alpha+1}{t^{\alpha+1}}\int_0^t f^{p_0}(s)s^\alpha ds\rr^{p/p_0}w(t)dt \nonumber \\
& \leq \varphi(1)^{p/p_0}\int_0^\infty\lr\frac{\alpha+1}{t^{\alpha+1}}\int_0^t g^{p_0}(s)s^\alpha ds\rr^{p/p_0}w(t)dt \nonumber \\
& = \varphi(1)^{p/p_0}\int_0^\infty\lr S_\psi g^{p_0}(t)\rr^{p/p_0}w(t)dt, \label{ineq a1}
\end{align}
with $\psi(s)= s^\alpha.$

Now, let $w \in QB_{\beta,\infty}.$ Then there exists $q>0$ such that $w\in QB_{\beta,q}.$ We can choose $\alpha>-1$ such that 
$q= (\alpha +1)\frac{p}{p_0}.$ Then $w\in QB_{\beta,(\alpha +1)\frac{p}{p_0}},$ which in view of $(\ref{eqn aa})$ implies that the following holds for all $r>0:$ 
\[
\int_r^\infty \lr\frac{r}{t}\rr^{(\alpha +1)\frac{p}{p_0}}w(t)dt \leq (C-1)\int_0^r \lr\frac{t}{r}\rr^{\beta(\alpha +1)\frac{p}{p_0}}w(t)dt
\]
or, equivalently
\[
\int_r^\infty \lr\frac{\Psi(r)}{\Psi(t)}\rr^{p/p_0}w(t)dt \leq (C-1)\int_0^r \lr\frac{\Psi(t)}{\Psi(r)}\rr^{\beta p/p_0}w(t)dt
\]
with $\psi(s)=s^\alpha$. But the last inequality, in view of Theorem A, holds if and only if 
\be \label{eqn 7}
\int_0^\infty \lr S_\psi g^{p_0}(t)\rr^{p/p_0}w(t)dt \leq \frac{C}{(\beta+1)^{p/p_0}}\int_0^\infty g^p(t)w(t)dt,
\ee
where $C=[w]_{QB_{\beta,(\alpha +1)\frac{p}{p_0}}}.$ 
Now $(\ref{ineq a1})$ and $(\ref{eqn 7})$ give that
\[
\int_0^\infty f^p(t)w(t)dt \leq [w]_{QB_{\beta,(\alpha +1)\frac{p}{p_0}}}\lr\frac{\varphi(1)}{\beta+1}\rr^{p/p_0}\int_0^\infty g^p(t)w(t)dt,
\]
so that on taking the infimum over all $\alpha > -1,$ the assertion follows.
\end{proof}

\section{Extrapolation results in grand Lebesgue spaces}
In this section, we shall prove a version of the extrapolation result (Theorem $\ref{thm 1}$) in the framework of grand Lebesgue spaces defined on finite intervals, which without any loss of generality is taken as $I=(0,1)$.\\

Let $0<p<\infty$ and $-1<\beta<\infty$. We say that a weight function $w$ on $I$ belongs to the class$QB_{\beta,p}(I)$ if there exists a constant $C>0$ such that the inequality:
\[
\int_r^1 \lr\frac{r}{t}\rr^p w(t)dt \leq C\int_0^r \lr\frac{t}{r}\rr^{\beta p}w(t)dt
\]
holds for all $0<r\leq 1.$ Also, for $0<r\leq 1,$ we set 
\[
[w]_{QB_{\beta,p}(I)}:= \inf \left \{C>1: \int_r^1 \lr\frac{r}{t}\rr^p w(t)dt \leq (C-1)\int_0^r \lr\frac{t}{r}\rr^{\beta p}w(t)dt \right \}.
\]

It can be seen that if $0<p<\infty$ and $w\in QB_{\beta,p}(I),$ then the function $\widetilde{w}= w \chi_I \in QB_{\beta,p}$ and 
\[
[w]_{QB_{\beta,p}(I)}=[\widetilde{w}]_{QB_{\beta,p}}.
\]
Following the arguments used in Lemma \ref{lemma b}, we can prove:
\begin{lemma} \label{prop 1}
Let $-1<\beta \leq 0$ and $1\leq p<\infty.$ If $x^\alpha \in QB_{\beta,p}(I),$ then there exists $0 <\varepsilon < p(\beta +1)$ such that $x^\alpha \in QB_{\beta,p-\varepsilon}(I).$
\end{lemma}

\begin{definition}
For a given $-1<\beta<\infty,$ a weight function $w\in \widehat{Q}B_{\beta,p}(I)$ if \\
(i) $w\in QB_{\beta,p}(I)$; and \\
(ii) there exists $0 < \varepsilon <p(\beta +1)$ such that $w\in QB_{\beta,p-\varepsilon}(I).$
\end{definition}
\begin{remark}
It can be checked that for $-1<\beta \leq 0,$ the power weights $x^\alpha \in QB_{\beta,p}(I)$ if and only if $-\beta p-1<\alpha<p-1.$ Then, in view of  Lemma $\ref{prop 1},$ the class $ \widehat{Q}B_{\beta,p}(I)$ is reasonably defined.
\end{remark}

It is seen that Theorem $\ref{thm 1}$ can be modified for the interval $I.$  We state it formally for later purpose.

\begin{theorem} \label{thm 2}
Let $\varphi \uparrow$ be defined on $\mathbb {R}^+$ and $(f,g)$ be a pair of functions such that $f,g \in Q_{\beta}(I), ~-1<\beta\leq 0$. Let $1 \le p_0<\infty$ and that for every weight function $w \in QB_{\beta,p_0}(I),$ the inequality 
\[
\int_0^1 f^{p_0}(x)w(x)dx \leq \varphi\lr[w]_{QB_{\beta,p_0}(I)}\rr\int_0^1 g^{p_0}(x)w(x)dx
\]
holds. Then for every $p_0 \leq p<\infty$ and every $w\in \widehat{Q}B_{\beta,p}(I),$ the following inequality holds:
\[
\int_0^1 f^p(x)w(x)dx \leq K'(p)\int_0^1 g^p(x)w(x)dx,
\]
where 
\[
K'(p):= \displaystyle\inf_{0<\delta <p_0(\beta+1)}[w]_{QB_{\beta,(p_0-\delta)\frac{p}{p_0}}(I)}\left[ \frac{1}{\beta+1}\lr \frac{p_0(\beta+1)-\delta}{p_0-\delta}\rr \varphi \lr\frac{p_0(\beta+1)}{\delta}\rr\right]^{p/p_0}.
\]
\end{theorem}
In this section, we shall prove Theorem $\ref{thm 1}$ in the framework of grand Lebesgue spaces $L^{p),\theta}(I)$ which consist of all measurable functions $f$ finite a.e. on $I$ for which 
\[
\|f\|_{L^{p),\theta}(I)}:= \sup_{0<\varepsilon<p-1}\lr \varepsilon^\theta\int_0^1 |f(t)|^{p-\varepsilon}dt\rr^{1/(p-\varepsilon)} < \infty.
\]

These spaces without weight have been defined in \cite{gis}, which in fact, were initially defined for $\theta=1$ by Iwaniec and Sbordone \cite{is} and later have been generalized, studied and applied by several people in different directions. We refer to 
\cite{pma1} and the references therein. For some very recent updates on grand Lebesgue spaces, we mention \cite{fk}, \cite{th}, \cite{jmsv}, \cite{vmk}, \cite{km1}, \cite{km2}, \cite{mnova}. \\

We now prove the  following:
\begin{theorem} \label{thm 3}
Let $\theta>0, ~\varphi$ be a non-negative $\uparrow$ function defined on $(0, \infty),~-1<\beta\leq 0, 1 < p_0<\infty$ and $(f,g)$ be a pair of functions such that $f,g \in Q_\beta (I)$. Suppose that for every $w\in QB_{\beta,p_0}(I)$, the following inequality holds:
\[
\int_0^1 f^{p_0}(x)w(x)dx \leq \varphi\lr[w]_{QB_{\beta,p_0}(I)}\rr\int_0^1 g^{p_0}(x)w(x)dx.
\]
Then for every $p:~p_0 \le p<\infty$  and every $w\in \widehat{Q}B_{\beta,p}(I),$ the inequality 
\[
\|f\|_{L_w^{p),\theta}(I)} \leq C^* \|g\|_{L_w^{p),\theta}(I)}
\]
holds with 
\[
C^* = \displaystyle \inf_{0<\sigma < p-1} \left [\max \left \{ 1,p^\theta\sigma^{-\frac{\theta}{p-\sigma}}\lr W(I) +1\rr^{\frac{p-1-\sigma}{p-\sigma}}\right\} \sup_{0<\varepsilon\leq\sigma}{(K'(p-\varepsilon))}^\frac{1}{p-\varepsilon} \right]. 
\]
\end{theorem}
\begin{proof}
Let $w\in \widehat{Q}B_{\beta,p}(I),$ then by definition $w\in QB_{\beta,p}(I),$ and there exists $0<\xi<p(\beta +1)$ such that $w\in QB_{\beta,p-\xi}(I).$ Take $\sigma = \min \{ \xi, p-p_0\}.$ Clearly $0< \sigma < p-1,$ so, by Remark \ref{rmk21}, $w\in QB_{\beta,p-\sigma}(I).$ Let $\varepsilon \in (0, \sigma).$ Then, in view of the fact that $\widehat{Q}B_{\beta,p} \subset \widehat{Q}B_{\beta,q}$ for $p < q,$ we have $w \in \widehat{Q}B_{\beta,p-\varepsilon}(I).$ Therefore, by Theorem $\ref{thm 2}$ we have
\be \label{eqn31}
\int_0^1 f^{p-\varepsilon}(x)w(x)dx \leq K'(p-\varepsilon)\int_0^1 g^{p-\varepsilon}(x)w(x)dx,
\ee
where 
\[
K'(p-\varepsilon) := \displaystyle\inf_{0<\delta <p_0(\beta+1)}[w]_{QB_{\beta,(p_0-\delta)\frac{p-\varepsilon}{p_0}}(I)}\left[ \frac{1}{\beta+1}\lr \frac{p_0(\beta+1)-\delta}{p_0-\delta}\rr \varphi \lr\frac{p_0(\beta+1)}{\delta}\rr\right]^{p-\varepsilon/p_0}.
\]
Now, for $\sigma<\varepsilon<p-1,$ using H\"older's inequality with the indices $\frac{p-\sigma}{p-\varepsilon}$ and $\frac{p-\sigma}{\varepsilon-\sigma},$ we obtain
\begin{align}
\|f\|_{L_w^{p-\varepsilon}}
& = \lr\int_0^1 f^{p-\varepsilon}(x)w(x)dx\rr^{1/{p-\varepsilon}} \nonumber \\
& \leq \lr\int_0^1 f^{p-\sigma}(x)w(x)dx\rr^{\frac{1}{p-\sigma}}\lr W(I) \rr^{\frac{\varepsilon-\sigma}{(p-\sigma)(p-\varepsilon)}} \nonumber \\
& \leq \lr\int_0^1 f^{p-\sigma}(x)w(x)dx\rr^{\frac{1}{p-\sigma}}\lr W(I) +1\rr^\frac{p-1-\sigma}{p-\sigma}.   \label{ineq z7}
\end{align}
Now in view of  $(\ref{eqn31})$ and  $(\ref{ineq z7}),$ we get
\begin{align*}
\|f\|_{L_w^{p),\theta}(I)}
&= \max \left \{ \sup_{0<\varepsilon\leq\sigma}\varepsilon^{\frac{\theta}{p-\varepsilon}}\|f\|_{L_w^{p-\varepsilon}},\sup_{\sigma<\varepsilon < p-1}\varepsilon^{\frac{\theta}{p-\varepsilon}}\|f\|_{L_w^{p-\varepsilon}}\right\}\\
& \leq \max \left \{ \sup_{0<\varepsilon\leq\sigma}\varepsilon^{\frac{\theta}{p-\varepsilon}}\|f\|_{L_w^{p-\varepsilon}},\sup_{\sigma<\varepsilon < p-1}\varepsilon^{\frac{\theta}{p-\varepsilon}}\|f\|_{L_w^{p-\sigma}}\lr W(I) +1\rr^\frac{p-1-\sigma}{p-\sigma}\right\}\\
& \leq \max \left \{ 1,p^\theta \sigma^{-\frac{\theta}{p-\sigma}}\lr W(I) +1\rr^\frac{p-1-\sigma}{p-\sigma}\right\}\sup_{0<\varepsilon\leq\sigma}\varepsilon^{\frac{\theta}{p-\varepsilon}}\|f\|_{L_w^{p-\varepsilon}}\\
& \leq \max \left \{ 1,p^\theta \sigma^{-\frac{\theta}{p-\sigma}}\lr W(I) +1\rr^{\frac{p-1-\sigma}{p-\sigma}}\right\}\sup_{0<\varepsilon\leq\sigma}\varepsilon^{\frac{\theta}{p-\varepsilon}}(K'(p-\varepsilon))^\frac{1}{p-\varepsilon}\|g\|_{L_w^{p-\varepsilon}}\\
& \leq \max \left \{ 1,p^\theta \sigma^{-\frac{\theta}{p-\sigma}}\lr W(I) +1\rr^{\frac{p-1-\sigma}{p-\sigma}}\right\}\sup_{0<\varepsilon\leq\sigma}(K'(p-\varepsilon))^\frac{1}{p-\varepsilon} \|g\|_{L_w^{p),\theta}(I)}\\
& = C^* \|g\|_{L_w^{p),\theta}(I)},
\end{align*}
where $C^* = c(p,\theta,\sigma) \displaystyle \sup_{0<\varepsilon \leq\sigma} (K'(p-\varepsilon))^\frac{1}{p-\varepsilon}$ and
\[
c(p,\theta,\sigma) = \max \left \{ 1,~p^\theta \sigma^{-\frac{\theta}{p-\sigma}}\lr W(I) +1\rr^{\frac{p-1-\sigma}{p-\sigma}}\right\}.
\]
The proof is completed.
\end{proof}

\section{Application}

We provide an application of the extrapolation result proved in the previous section to characterize the boundedness of the Hardy averaging operator $H$
between weighted grand Lebesgue spaces $L_w^{p),\theta}(I)$ for quasi-monotone functions. We prove the following:
\begin{theorem} \label{thm 4}
Let $1< p<\infty, -1<\beta\leq 0$ and $\theta>0.$ The inequality
\be \label{eqn z} 
\|Hf\|_{L_w^{p),\theta}(I)} \leq C\|f\|_{L_w^{p),\theta}(I)}
\ee
holds for all $f \in Q_\beta (I)$ if and only if $w\in \widehat{Q}B_{\beta,p}(I).$
\end{theorem}
\begin{proof}
Let us first assume that $w\in \widehat{Q}B_{\beta,p}(I).$
Note that if $f\in Q_\beta(I),$ then for $0<t\leq s$ and $\alpha\in I,$ we have that
\[
t^{-\beta}f\lr\frac{\alpha t}{s}\rr \geq s^{-\beta}f(\alpha)
\] 
using which we get that 
\begin{align*}
s^{-\beta}Hf(s)
& \leq \frac{1}{s}\int_0^s t^{-\beta}f\lr\frac{\alpha t}{s}\rr d\alpha\\
& = t^{-\beta-1}\int_0^t f(z)dz\\
& = t^{-\beta}Hf(t)
\end{align*}
i.e., $Hf \in Q_\beta (I).$

Further, on taking $\psi \equiv 1$ in a modified form of Theorem A, and considering the functions $f$ defined on $I$ instead of $(0,\infty),$  we see that the inequality 
\[
\int_0^1 \lr Hf(x)\rr^p w(x)dx \leq C\int_0^1 f^p(x) w(x)dx
\]
holds. Now, in view of Theorem $\ref{thm 3},$ the inequality $(\ref{eqn z})$ holds.\\

Conversely, assume that the inequality $(\ref{eqn z})$ holds. Consider the test function $f_r(x)= x^\beta \chi_{(0,r)}(x)$ for $0<r<1.$ Then
\begin{align} 
\|f_r\|_{L_w^{p),\theta}(I)} &= \sup_{0<\varepsilon < p-1} \lr \varepsilon^\theta \int_0^r x^{\beta (p-\varepsilon)}w(x)dx\rr^{\frac{1}{p-\varepsilon}} \nonumber \\
&= \max \left \{ \sup_{0<\varepsilon\leq\sigma}\varepsilon^{\frac{\theta}{p-\varepsilon}}\|f_r\|_{L_w^{p-\varepsilon}},\sup_{\sigma<\varepsilon < p-1}\varepsilon^{\frac{\theta}{p-\varepsilon}}\|f_r\|_{L_w^{p-\varepsilon}}\right\}, \nonumber 
\end{align}
where $\sigma$ is chosen such that $0 < \sigma < \displaystyle \min \{(\beta + 1)p,~p-1\}.$ Now, for $\sigma < \varepsilon < p-1,$ taking the conjugate indices $\frac{p-\sigma}{p-\varepsilon}$ and $\frac{p-\sigma}{\varepsilon-\sigma},$ on using H\"older's inequality  we obtain
\begin{align} \label{eqn042}
\|f_r\|_{L_w^{p-\varepsilon}}
& \leq \lr\int_0^r x^{\beta (p-\sigma)} w(x)dx\rr^\frac{1}{p-\sigma} \lr W(I) \rr^\frac{\varepsilon-\sigma}{(p-\sigma)(p-\varepsilon)}  \nonumber \\
& \leq \lr\int_0^r x^{\beta (p-\sigma)} w(x)dx \rr^\frac{1}{p-\sigma} \lr W(I) +1\rr^\frac{p-1-\sigma}{p-\sigma}.   
\end{align}
Thus, on using (\ref{eqn042}) and an argument from (\cite{amksh}, Theorem 3.1), we have
\begin{align} \label{eq44}
\|f_r\|_{L_w^{p),\theta}(I)}
& \leq \max \left \{ 1,p^\theta \sigma^{-\frac{\theta}{p-\sigma}}\lr W(I) +1 \rr^\frac{p-1-\sigma}{p-\sigma}\right\}\sup_{0<\varepsilon\leq\sigma} \varepsilon^{\frac{\theta}{p-\varepsilon}}\|f_r\|_{L_w^{p-\varepsilon}} \nonumber \\
& = \max \left \{ 1,p^\theta \sigma^{-\frac{\theta}{p-\sigma}}\lr W(I) +1\rr^{\frac{p-1-\sigma}{p-\sigma}}\right\} \varepsilon_r^{\frac{\theta}{p-\varepsilon_r}}\|f_r\|_{L_w^{p-\varepsilon_r}} \nonumber \\
& = C_1 \lr \varepsilon_r^\theta \int_0^r x^{\beta (p-\varepsilon_r)}w(x)dx\rr^{\frac{1}{p-\varepsilon_r}} 
\end{align}
for some $0 < \varepsilon_r \le \sigma,$ where $C_1 := \displaystyle \inf_{0<\sigma < (\beta +1)p}  \max \left \{ 1,p^\theta \sigma^{-\frac{\theta}{p-\sigma}}\lr W(I) +1\rr^{\frac{p-1-\sigma}{p-\sigma}}\right\}.$
Further, note that 
\begin{align*}
\int_0^1 \lr Hf_r(x)\rr^{p-\varepsilon}w(x)dx
& \geq \int_r^1 \lr Hf_r(x)\rr^{p-\varepsilon}w(x)dx \\ 
&= \lr\frac{r^{\beta+1}}{\beta+1}\rr^{p-\varepsilon}\int_r^1 \frac{w(x)}{x^{p-\varepsilon}}dx
\end{align*}
so that
\begin{align*}
\|Hf_r\|_{L_w^{p),\theta}(I)}
& \geq \frac{r^{\beta+1}}{\beta+1}\sup_{0< \varepsilon < p-1}\lr \varepsilon^\theta \int_r^1\frac{w(x)}{x^{p-\varepsilon}}dx\rr^{\frac{1}{p-\varepsilon}}\\
& \geq \frac{r^{\beta+1}}{\beta+1}\lr{\varepsilon_r}^\theta \int_r^1\frac{w(x)}{x^{p-\varepsilon_r}}dx\rr^{\frac{1}{p-\varepsilon_r}}.
\end{align*}
The above estimate together with (\ref{eq44}), and the assumption that $(\ref{eqn z})$ holds, gives that 
\[
\frac{r^{\beta+1}}{\beta+1}\lr{\varepsilon_r}^\theta \int_r^1\frac{w(x)}{x^{p-\varepsilon_r}}dx\rr^{\frac{1}{p-\varepsilon_r}} \leq CC_1 \lr{\varepsilon_r}^\theta \int_0^r x^{\beta(p-\varepsilon_r)}w(x)dx\rr^{\frac{1}{p-\varepsilon_r}}.
\]
Therefore,
\begin{align*}
\int_r^1 \lr\frac{r}{x}\rr^{p-\varepsilon_r}w(x)dx &\leq \lr CC_1 (\beta+1)\rr^{p-\varepsilon_r}\int_0^r \lr\frac{x}{r}\rr^{\beta(p-\varepsilon_r)}w(x)dx \\
&\leq \lr CC_1 (\beta+1) +1\rr^p\int_0^r \lr\frac{x}{r}\rr^{\beta(p-\varepsilon_r)}w(x)dx.
\end{align*}
Thus, $w \in QB_{\beta, p-\varepsilon_r}(I),$ where $0< \varepsilon_r < (\beta+1)p.$  Consequently, $w \in QB_{\beta, p}(I)$ and hence $w \in \widehat{Q}B_{\beta, p}(I).$
\end{proof}
 
\bigskip
\noindent {\it Acknowledgment.} The first author acknowledges the MATRICS Research Grant No. MTR/2019 /000783 of SERB, Department of Science and Technology (DST), India. Also, the second author acknowledges the research fellowship award No.:  09/045(1716)/2019-EMR-I of Council of Scientific and Industrial Research (CSIR), INDIA. \\

\noindent {\it Conflict of interest statement.} On behalf of all authors, the corresponding author states that there is no conflict of interest.

%\bigskip

%\bigskip
%\newpage

%\vspace{15pt}

%\noindent Department of Mathematics\\
%Dyal Singh College (University of Delhi)\\
%Lodhi Road, Delhi - 110003, India\\
%(Email : arunpalsingh@dsc.du.ac.in)

%\bigskip

%\noindent  Department of Mathematics\\
%University of Delhi, Delhi - 110007, India\\
%(Email: drrpanchal0@gmail.com)

%\bigskip

%\noindent  Department of Mathematics\\
%South Asian University\\
%Akbar Bhawan, Chanakya Puri, \\
%New Delhi - 110 021, India\\
%(Email : pankaj.jain@sau.ac.in; pankajkrjain@hotmail.com)

%\bigskip

%\noindent Department of Mathematics\\
%Lady Shri Ram College for Women (University of Delhi)\\
%Lajpat Nagar, Delhi - 110024, India\\
%(Email : monikasingh@lsr.du.ac.in)

\end{document}